\documentclass[a4paper,12pt]{amsart}

\usepackage[english]{babel}
\usepackage[utf8]{inputenc}
\allowdisplaybreaks  

\usepackage{amssymb,xypic,amscd}

\usepackage{mathtools}
\usepackage{etoolbox}

\usepackage[colorlinks,linkcolor=Red3,citecolor=Red3,urlcolor=Cyan3]{hyperref}

\usepackage{enumitem}

\usepackage{parskip}

\usepackage[svgnames,x11names,dvipsnames]{xcolor}

\newtheorem{definition}[equation]{Definition}
\newtheorem{theorem}[equation]{Theorem}

\newtheorem{corollary}[equation]{Corollary}
\newtheorem{lemma}[equation]{Lemma}
\newtheorem{proposition}[equation]{Proposition}

\theoremstyle{remark}

\theoremstyle{remark}

\numberwithin{equation}{section} 

\newcommand\Spec{\operatorname{Spec}}

\renewcommand\emptyset{\varnothing}

\DeclarePairedDelimiterXPP\globalsymbol[1]{}{\langle}{\rangle}{}{\ifblank{#1}{\cdot,\cdot}{#1}}
\newcommand{\normSymbol}{N_{k(x)/k}}
\DeclarePairedDelimiterXPP\norm[1]{\normSymbol}{\lparen}{\rparen}{}{\ifblank{#1}{\cdot}{#1}}
\DeclarePairedDelimiterXPP\localsymbol[1]{}{\lparen}{\rparen}{}{\ifblank{#1}{\cdot,\cdot}{#1}}
\DeclarePairedDelimiterXPP\globalCommutator[1]{}{\lbrace}{\rbrace}{\raisebox{0pt}[1ex][1ex]{$^{\mathbb{A}_{X}}_{\mathbb{A}^{+}_{X}}$}}{\ifblank{#1}{\cdot,\cdot}{#1}}
\DeclarePairedDelimiterXPP\localCommutator[1]{}{\lbrace}{\rbrace}{^{K_{x}}_{\widehat{\O}_{X,x}}}{\ifblank{#1}{\cdot,\cdot}{#1}}

\renewcommand{\O}{{\mathcal{O}}}

\newcommand{\A}{{\mathbb{A}}}

\renewcommand{\P}{{\mathcal{P}}}

\newcommand{\Z}{{\mathbb{Z}}}

\newcommand{\Ucal}{\mathcal{U}}
\newcommand{\Q}{\mathbb{Q}}
\newcommand{\m}{\mathfrak{m}}

\renewcommand\tilde{\widetilde}

\makeatletter
\@namedef{subjclassname@2020}{%
  \textup{2020} Mathematics Subject Classification}
\makeatother

\title{Prime spectrum of the ring of adeles of a number field}
\author[Álvaro Serrano Holgado]{Álvaro Serrano Holgado\\
\\
\tiny{\emph{Department of Mathematics, University of Salamanca}}}

\email{Alvaro\_\,Serrano@usal.es}



\begin{document}

\maketitle

\begin{abstract}
Much is known about the adele ring of an algebraic number field from the perspective of Harmonic Analysis and Class Field Theory. However, its ring-theoretical aspects are often ignored. Here we present a description of the prime spectrum of this ring and study some of the algebraic and topological properties of these prime ideals. We also study how they behave under separable extensions of the base field and give an indication of how this study can be applied in adele rings not of number fields.
\end{abstract}


\section{Introduction}
\label{sec:intro}

Let $K$ be an algebraic number field, i.e., a finite field extension of the rational numbers $\Q$, and denote by $\mathfrak{o}_K$ the ring of algebraic integers in $K$, namely, the algebraic closure of $\Z$ in $K$.

Consider the set $X$ of places of $K$, defined as the set of equivalence classes of absolute values on $K$, where two absolute values are equivalent if one is a power of the other. We may partition $X$ into $X_\infty$, the set of Archimedean places, and $X_f$, the set of non-Archimedean or \emph{finite} places. For $\upsilon\in X$, let $K_{\upsilon}$ be the completion of $K$ with respect to any absolute value in the class of $\upsilon$, and for $\upsilon\in X_f$, let $\mathfrak{o}_{\upsilon}$ be the valuation ring of $\upsilon$, with maximal ideal $\mathfrak{m}_{\upsilon}$. We have a one-to-one correspondence between $X_{f}$ and the set of non-zero (and therefore maximal) prime ideals of $\mathfrak{o}_K$, whereby, if the place $\upsilon$ corresponds to the prime ideal $\mathfrak{p}$, then $\mathfrak{o}_{\upsilon}$ is isomorphic to the $\mathfrak{p}$-adic completion of $\mathfrak{o}_{K}$.

For a fixed finite subset $S_0 \subseteq X$ containing $X_\infty$, let
\[
	\Lambda=\{S\subseteq X \:\vert\: S\text{ is finite and }S_0\subseteq S\}.
\]
For $S\in\Lambda$, consider the product ring
\[
	\A_{K,S}=\underset{\upsilon\in S} \prod K_{\upsilon}
	\: \times \!\!
	\underset{\upsilon\in X\setminus S}\prod \!\! \mathfrak{o}_{\upsilon}.
\]
If $S_1\subseteq S_2$ with $S_1,S_2\in\Lambda$, there is a natural inclusion of rings $\A_{K,S_1}\hookrightarrow \A_{K,S_2}$ and with respect to these inclusions, the family $\{\A_{K,S}\}_{S\in\Lambda}$ is a direct system of rings. The \emph{adele ring} of $K$, which we will denote by $\A_K$, may be defined as the direct limit
\[
	\A_K=\underset{\underset{S\in\Lambda} \longrightarrow }\lim \:\: \A_{K,S}.
\]
This definition is equivalent to the usual definition of $\A_{K}$ as the restricted direct product of the fields $K_{\upsilon}$ with respect to the subrings $\mathfrak{o}_{\upsilon}$, namely the subring of $\prod K_{\upsilon}$ given by
\[
	\A_K=\left\{\alpha=(\alpha_{\upsilon})_{\upsilon\in X} \in \prod K_{\upsilon} \:\vert\: \alpha_{\upsilon}\in\mathfrak{o}_{\upsilon}\text{ for almost all }\upsilon\in X\right\},
\]

and it is clear that the definition does not depend on the choice of $S_0$: any other choice will differ from the given $S_0$ in at most a finite number of places, and this doesn't change the direct limit nor the restricted direct product.

A comprehensive study of the adele ring necessitates considering it as a topological ring. Each $K_{\upsilon}$ has the complete metric topology defined by the place $\upsilon$, and at finite places, $\mathfrak{o}_{\upsilon}$ has the induced non-archimedean topology. Giving each $\A_{K,S}$ the product topology, $\A_K$ naturally has the direct limit topology, making it a Hausdorff and locally compact topological ring. Standard references for adeles and their applications are \cite{CaFr}, \cite{Ram} or \cite{TaPhD}.

\section{Prime ideals in $\A_{K,S}$}
\label{sec:primesproduct}

To compute the prime spectrum of $\A_K$ we will first describe $\Spec(\A_{K,S})$ for $S\in\Lambda$ and then study which of these prime ideals are compatible with the morphisms in the direct system. Note that
\[
	\A_{K,S}=K_S\times \mathfrak{o}_S,
\]
where
\[
	K_S=\underset{\upsilon\in S}\prod K_{\upsilon}
	\qquad\text{ and }\qquad
	\mathfrak{o}_S=\underset{\upsilon\notin S}\prod \mathfrak{o}_{\upsilon}.
\]
The prime ideals in a product of two rings $A\times B$ are the ideals of the form $\mathfrak{p}\times B$, with $\mathfrak{p}\in\Spec(A)$, or $A\times\mathfrak{q}$, where $\mathfrak{q}\in\Spec(B)$. A prime ideal of this form is maximal (or minimal) if and only if the corresponding factor $\mathfrak{p}$ or $\mathfrak{q}$ is maximal (or minimal). This is easily generalized to finite products, such as $K_S$, so once we describe $\Spec(\mathfrak{o}_S)$ we will also obtain $\Spec(\A_{K,S})$.

We note that $\mathfrak{o}_S$ is an infinite product of Prüfer domains, so we can apply the results in \cite{Fino} to describe $\Spec(\mathfrak{o}_S)$. This involves the use of ultrafilters.

\begin{definition}\label{dfn:ultrafilter}
Let $(\mathcal{B},\wedge,\vee)$ a Boolean algebra. A nonempty subset $\Ucal\subseteq\mathcal{B}$ is called an \textbf{ultrafilter} on $\mathcal{B}$ if it satisfies the following conditions:
\begin{enumerate}
\item $0\notin\Ucal$.

\item If $A,B\in\Ucal$ then $A\wedge B\in\Ucal$.

\item If $A\in\Ucal$ and $A\leq B$, where $\leq$ is the canonical order relation on $\mathcal{B}$, then $B\in\Ucal$.

\item Given $A\in\mathcal{B}$, either $A\in\Ucal$ or $\neg A\in\Ucal$.
\end{enumerate}

If $\Omega$ is a set, an ultrafilter on $\Omega$ is an ultrafilter on its power set with the usual Boolean algebra structure $(\mathcal{P}(\Omega),\cap,\cup)$.
\end{definition}

An ultrafilter $\Ucal$ on $\Omega$ is called \textbf{principal} if there is an $\omega_0\in\Omega$ such that $\omega_0\in U$ for every $U\in\Ucal$. In this case, it is easy to check that $\Ucal$ consists exactly of the subsets of $\Omega$ containing $\omega_0$, and we say that it is the principal ultrafilter generated by $\omega_0$. One can also prove that an ultrafilter on $\Omega$ is principal if and only if it contains some finite subset of $\Omega$. Therefore, an ultrafilter is non-principal (or \textit{free}) if and only if it contains every cofinite subset of $\Omega$.

We will be considering ultrafilters on the set $\Omega_S=X\setminus S$. By~\cite{Fino}, for every ultrafilter $\Ucal$ on $\Omega_S$, the set
\[
	\mathfrak{M}_{S,\:\Ucal}=\left\{ \alpha\in\mathfrak{o}_S \:\vert\: \{\upsilon\in\Omega_S\:\vert\: \alpha_{\upsilon}\in\m_{\upsilon}\}\in\Ucal\right\}
\]
is a maximal ideal in $\mathfrak{o}_S$. This gives us a one-to-one correspondence between ultrafilters on $\Omega_S$ and maximal ideals of $\mathfrak{o}_S$.
For the principal ultrafilter $\Ucal=\Ucal_{\upsilon_0}$ corresponding to the place $\upsilon_0\in \Omega_S$, we have
\[
	\mathfrak{M}_{S,\:\Ucal_{\upsilon_0}}=\mathfrak{M}_{S,\upsilon_0}=\{\alpha\in\mathfrak{o}_S \:\vert\: \alpha_{\upsilon_0}\in\m_{\upsilon_0}\}.
\]
Minimal prime ideals in $\mathfrak{o}_S$ also correspond to ultrafilters on $\Omega_S$: given an ultrafilter $\Ucal$, we assign to it the minimal prime ideal
\[
	\mathfrak{n}_{S,\:\Ucal}=\left\{\alpha\in\mathfrak{o}_S \:\vert\: \{\upsilon\in \Omega_S \:\vert\: \alpha_{\upsilon}=0\}\in\Ucal \right\}.
\]
For the principal ultrafilter $\Ucal=\Ucal_{\upsilon_0}$, we now have
\[
	\mathfrak{n}_{S,\:\Ucal_{\upsilon_0}}=\mathfrak{n}_{S,\upsilon_0}=\{\alpha\in\mathfrak{o}_S \:\vert\: \alpha_{\upsilon_0}=0\}.
\]

For any ultrafilter $\Ucal$, $\mathfrak{n}_{S,\:\Ucal}$ is the only minimal prime ideal contained in the maximal ideal $\mathfrak{M}_{S,\:\Ucal}$, and $\mathfrak{M}_{S,\:\Ucal}$ is the only maximal ideal containing the minimal prime ideal $\mathfrak{n}_{S,\:\Ucal}$. Therefore, any prime ideal in $\mathfrak{o}_S$ contains exactly one minimal prime ideal and is contained in exactly one maximal ideal.

Given an ultrafilter $\Ucal$ on $\Omega_S$ and an element $\beta=(\beta_{\upsilon})\in(\Ucal)^{\mathfrak{o}_S}$, the set
\[
	\mathfrak{p}_{S,\:\Ucal}^{\beta}=
	\{\alpha\in\mathfrak{o}_S \:\vert\: \exists Y\in\Ucal,\: n\in\mathbb{N}\text{ s.t. } \upsilon(\alpha_{\upsilon}^n)\geq \upsilon(\beta_{\upsilon})\:\:\forall \upsilon\in Y\}
\]
is a prime ideal of $\mathfrak{o}_S$ contained in $\mathfrak{M}_{S,\:\Ucal}$ (note that this definition and notation is not the same as that of \cite{Fino}, but it is easy to check that for this case, in which all the rings in the product are discrete valuation rings, both definitions are the same). It is also the smallest prime ideal in $\mathfrak{M}_{S,\:\Ucal}$ that contains $\beta$. Furthermore, if $\Ucal=\Ucal_{\upsilon_0}$ for some $\upsilon_0\in S$, it follows from the definition that
\[
	\mathfrak{p}_{S,\:\Ucal_{\upsilon_0}}^{\beta}=\begin{cases}
	\mathfrak{M}_{S,\upsilon_0}\quad\text{ if }\beta_{\upsilon_0}\in \mathfrak{m}_{\upsilon_0}\setminus\{0\} &\\
	\mathfrak{n}_{S,\upsilon_0}\quad\text{ if }\beta_{\upsilon_0}\in\mathfrak{o}_{\upsilon_0}^*\cup\{0\}\end{cases}
\]
Now, if $\mathfrak{p}$ is any prime ideal of $\mathfrak{o}_S$ contained in $\mathfrak{M}_{S,\:\Ucal}$, we have that
\[
	\mathfrak{p}=\underset{\beta\in\mathfrak{p}}\bigcup \mathfrak{p}_{S,\:\Ucal}^\beta
\]
This shows that for $\upsilon_0\in \Omega_S$, the only prime ideal strictly contained in $\mathfrak{M}_{S,\upsilon_0}$ is $\mathfrak{n}_{S,\upsilon_0}$.

We can use the previous remarks to describe $\Spec(\A_{K,S})$ completely:

\begin{theorem}\label{thm:primeidealsinAXS}
The prime ideals in $\A_{K,S}$ are the following:
\begin{enumerate}
\item For $\upsilon_0\in X$, the set
\[
	M_{S,\upsilon_0}^{(0)}=\{\alpha\in\A_{K,S} \:\vert\: \alpha_{\upsilon_0}=0\},
\]
which is both a maximal and minimal prime ideal if $\upsilon_0\in S$ and only a minimal prime ideal if $\upsilon_0\in\Omega_S$. It is a principal ideal generated by $\alpha=(\alpha_{\upsilon})$ where $\alpha_{\upsilon}\in\mathfrak{o}_{\upsilon}^*$ if $\upsilon\neq \upsilon_0$ and $\alpha_{\upsilon_0}=0$.

\item If $\upsilon_0\in\Omega_S$, the set
\[
	M_{S,\upsilon_0}^{\m}=\{\alpha\in\A_{K,S} \:\vert\: \alpha_{\upsilon_0}\in\m_{\upsilon_0}\},
\]
which is a maximal prime ideal generated by $\alpha=(\alpha_{\upsilon})$ where $\alpha_{\upsilon}\in\mathfrak{o}_{\upsilon}^*$ if $\upsilon\neq\upsilon_0$ and $\alpha_{\upsilon_0}$ is a generator of $\m_{\upsilon_0}$.

\item If $\Ucal$ is a non-principal ultrafilter in $\Omega_S$, we have the maximal prime ideal
\[
	M_{S,\:\Ucal}^{\m}=\{\alpha\in\A_{K,S} \:\vert\: \{\upsilon\in\Omega_S\:\vert\: \alpha_{\upsilon}\in\m_{\upsilon}\}\in\Ucal\}
\]
and the minimal prime ideal
\[
	M_{S,\:\Ucal}^{(0)}=\{\alpha\in\A_{K,S} \:\vert\: \{\upsilon\in\Omega_S\:\vert\: \alpha_{\upsilon}=0\}\in\Ucal\}
\]
Neither of these are principal ideals.

\item If $\mathfrak{p}$ is any other prime ideal of $\A_{K,S}$, there is exactly one non-principal ultrafilter $\Ucal$ in $\Omega_S$ such that $\mathfrak{p}\subseteq M_{S,\:\Ucal}^{\m}$, and
\[
	\mathfrak{p}=\underset{\beta\in\mathfrak{p}}\bigcup \mathfrak{p}_{S,\:\Ucal}^\beta,
\]
whith $\mathfrak{p}_{S,\:\Ucal}^\beta$ defined as before. These are also not principal ideals.
\end{enumerate}
\end{theorem}

\section{Prime ideals in a direct limit}
\label{sec:primeslimit}



Let $\{R_i,\varphi_{ij}\}_{i,j\in I}$ be a direct system of rings and $R$ its direct limit, with $\varphi_i:R_i\rightarrow R$ the canonical morphisms.

\begin{definition}\label{dfn:directideals}
A \textbf{direct system of ideals} in $\{R_i,\varphi_{ij}\}$ is a family $\{\mathfrak{a}_i\}_{i\in I}$ such that $\mathfrak{a}_i$ is an ideal of $R_i$ for every $i\in I$ and $\varphi_{ij}^{-1}(\mathfrak{a}_j)=\mathfrak{a}_i$ whenever $i\leq j$.
\end{definition}

The direct systems of ideals of $\{R_i,\varphi_{ij}\}$ are in one-to-one correspondence with the ideals of $R$: for an ideal $\mathfrak{a}\subseteq R$, we have the direct system $\{\mathfrak{a}_i=\varphi_i^{-1}(\mathfrak{a})\}$, and for a direct system $\{\mathfrak{a}_i\}$, the direct limit
\[
	\mathfrak{a}=\underset{\underset{i\in I}\longrightarrow}\lim\: \mathfrak{a}_i
\]
is an ideal of $R$ giving rise to this direct system.

It follows immediately that $\mathfrak{a}$ is a prime ideal in $R$ if an only if every $\mathfrak{a}_i$ is a prime ideal of $R_i$. That is:

\begin{proposition}\label{pro:inverselimitspec}
If $\{R_i,\varphi_{ij}\}$ is a direct system of commutative rings and $R$ its direct limit (with morphisms $\varphi_i:R_i\rightarrow R$), $\{\Spec(R_i),\varphi_{ij}^*\}$ is an inverse system of topological spaces, and its inverse limit is $\Spec(R)$, with morphisms $\varphi_i^*:\Spec(R)\rightarrow\Spec(R_i)$.
\end{proposition}

\begin{definition}\label{dfn:uppmaxmin}
A direct system of prime ideals $\{\mathfrak{p}_i\}$ in $\{R_i,\varphi_{ij}\}$ is \textbf{upper-maximal} if for every $i\in I$ there is some $j\geq i$ such that $\mathfrak{p}_j$ is a maximal prime ideal in $R_j$. It is \textbf{upper-minimal} if for every $i\in I$ there is some $j\geq i$ such that $\mathfrak{p}_j$ is a minimal prime ideal in $R_j$.
\end{definition}

\begin{proposition}\label{pro:uppmaxmin}
If $\{\mathfrak{p}_i\}$ is an upper-maximal (resp., upper-minimal) direct system of prime ideals in $\{R_i,\varphi_{ij}\}$, its limit $\mathfrak{p}\subseteq R$ is a maximal (resp., minimal) prime ideal.
\end{proposition}

\begin{proof}
We will prove the maximal case. Let $\m$ be a maximal ideal in $R$ containing $\mathfrak{p}$ and take $\mathfrak{q}_i=\varphi_i^{-1}(\m)$ for every $i\in I$, which obviously contains $\mathfrak{p}_i$. For every $i\in I$ there is some $j\geq i$ for which $\mathfrak{p}_j$ is maximal in $R_j$, therefore $\mathfrak{p}_j=\mathfrak{q}_j$. It follows that $\mathfrak{p}_i=\mathfrak{q}_i$ for every $i\in I$, and therefore $\mathfrak{p}=\m$.
\end{proof}

\section{Prime ideals in the adele ring}
\label{sec:primesadeles}

We now know that $\Spec(\A_K)$ is the inverse limit of the inverse system of topological spaces $\{\Spec(\A_{K,S})\}$. Since we have already computed $\Spec(\A_{K,S})$ for each $S\in\Lambda$, we only have to see which of those ideals form direct systems of prime ideals.

\begin{lemma}\label{lem:ultrafiltersfinitepoints}
Let $\Omega$ be a set, $F\subseteq\Omega$ a finite subset and $\Omega_F=\Omega\setminus F$. There is a one-to-one correspondence between the ultrafilters on $\Omega_F$ and the ultrafilters on $\Omega$ not generated by a point in $F$.
\end{lemma}

\begin{proof}
If $\tilde{\Ucal}$ is an ultrafilter on $\Omega_F$, then
\[
	\Ucal=\{Y\subseteq\Omega\:\vert\: Y\cap\Omega_F\in\tilde{\Ucal}\}
\]
is an ultrafilter on $\Omega$ which is principal or non-principal if the same holds true for $\tilde{\Ucal}$. If they are principal, they have the same generator, so $\Ucal$ is not generated by a point in $F$.

If $\Ucal$ is now an ultrafilter on $\Omega$ not generated by a point in $F$, then
\[
	\Ucal_{F}=\{Y\cap\Omega_F \:\vert\: Y\in\Ucal\}
\]
is an ultrafilter on $\Omega_F$. These mappings give us the desired correspondence.
\end{proof}

We can apply lemma \ref{lem:ultrafiltersfinitepoints} to the situation $\Omega_{S_2}\subseteq\Omega_{S_1}$, where $S_1,S_2\in\Lambda$ and $S_1\subseteq S_2$. However, it will only help us in the case of ideals given by non-principal ultrafilters, since for any $\upsilon_0\in X$ and any $S\in\Lambda$ there will always be $S\subseteq S'$ with $\upsilon_0\in S'$ (any $S'\in\Lambda$ containing $S\cup\{\upsilon_0\}$).

Let us now look for direct systems of prime ideals among the principal prime ideals.

For any $\upsilon_{0}$, the ideals $\{M_{S,\upsilon_0}^{(0)}\}_{S\in\Lambda}$ form a direct system of prime ideals which is both upper-maximal and upper-minimal. Its direct limit is the ideal of $\A_K$
\[
	M_{\upsilon_0}=\{\alpha\in\A_K \:\vert\: \alpha_{\upsilon_0}=0\},
\]
which is both a maximal and minimal prime ideal, principal and generated by $\alpha=(\alpha_{\upsilon})$, where $\alpha_{\upsilon}\in\mathfrak{o}_{\upsilon}^*$ if $\upsilon\neq\upsilon_0$ and $\alpha_{\upsilon_0}=0$.

If $S\in\Lambda$ and $\upsilon_0\in\Omega_S$, the ideal $M_{S,\upsilon_0}^{\m}$ of $\A_{K,S}$ cannot be part of a direct system of ideals, since for $S'=S\cup\{\upsilon_0\}$, there is no prime ideal in $\A_{K,S'}$ that restricts to $M_{S,\upsilon_0}^{\m}$ in $\A_{K,S}$.

Let us now consider a non-principal ultrafilter $\Ucal_{S_0}$ in $\Omega_{S_0}$, and for any $S\in\Lambda$, $\Ucal_{S}$ the unique ultrafilter in $\Omega_S$ corresponding to $\Ucal_{S_0}$ via the equivalence given in Lemma~\ref{lem:ultrafiltersfinitepoints}. We have the direct systems of prime ideals $\{M_{\Ucal_S,S}^{\m}\}_{S\in\Lambda}$ and $\{M_{\Ucal_S,S}^{(0)}\}_{S\in\Lambda}$, which are upper-maximal and upper-minimal, respectively. They give us the following maximal and minimal prime ideals of $\A_K$:
\[
	M_{\Ucal_{S_0}}=\{\alpha=(\alpha_{\upsilon})\in\A_K \:\vert\: \{\upsilon\in\Omega_{S_0}\:\vert\:\alpha_{\upsilon}\in\m_{\upsilon}\}\in\Ucal_{S_0}\}
\]
and
\[
	m_{\Ucal_{S_0}}=\{\alpha=(\alpha_{\upsilon})\in\A_K \:\vert\: \{\upsilon\in\Omega_{S_0}\:\vert\:\alpha_{\upsilon}=0\}\in\Ucal_{S_0}\}.
\]
It follows that any other prime ideal in $\A_K$ will be between $m_{\Ucal}$ and $M_{\Ucal}$ for some non-principal ultrafilter $\Ucal$ in $\Omega_{S_0}$. This proves as well that we have already found all maximal and minimal prime ideals in $\A_K$.

\begin{proposition}\label{pro:caracprimesAX}
Take $S\in\Lambda$, $\Ucal$ an ultrafilter on $\Omega_{S_0}$,  $\Ucal_S$ the corresponding ultrafilter on $\Omega_S$, $\tilde{\mathfrak{p}}$ a prime ideal in $\A_{K,S}$ inside $M_{S,\:\Ucal_S}^{\m}$ and $\mathfrak{p}=\tilde{\mathfrak{p}}\cap\A_{K,S_0}$. Then, $\tilde{\mathfrak{p}}$ is completely determined by $\mathfrak{p}$ as follows:
\[
	\tilde{\mathfrak{p}}=\underset{\beta\in\mathfrak{p}}\bigcup \mathfrak{p}_{S,\:\Ucal_S}^\beta
\]
\end{proposition}

\begin{proof}
Since $\mathfrak{p}\subseteq\tilde{\mathfrak{p}}$ and
\[
	\tilde{\mathfrak{p}}=\underset{\beta\in\tilde{\mathfrak{p}}}\bigcup \mathfrak{p}_{S,\:\Ucal_S}^\beta,
\]
it is clear that
\[
	\underset{\beta\in\mathfrak{p}}\bigcup \mathfrak{p}_{S,\:\Ucal_S}^\beta\subseteq\tilde{\mathfrak{p}}.
\]
Now take $\tilde{\beta}\in\tilde{\mathfrak{p}}$ and $\tilde{Y}=\{\upsilon\in\Omega_S \:\vert\: \tilde{\beta}_{\upsilon}\in\m_{\upsilon}\}\in \Ucal_S\subseteq\Ucal$. We have that $\mathfrak{p}_{S,\:\Ucal_S}^{\tilde{\beta}}\subseteq\tilde{\mathfrak{p}}$, and by the definition of $\mathfrak{p}_{S,\:\Ucal_S}^{\tilde{\beta}}$, the element $\beta\in\A_{K,S}$ defined by $\beta_{\upsilon}=\tilde{\beta}_{\upsilon}$ if $\upsilon\in\Omega_S$ and $\beta_{\upsilon}=0$ if $\upsilon\notin S$ is also in $\mathfrak{p}_{S,\:\Ucal_S}^{\tilde{\beta}}$. Furthermore $\tilde{\beta}\in \mathfrak{p}_{S,\:\Ucal_S}^{\beta}$ by the same reasoning. Therefore,
\[
	\mathfrak{p}_{S,\:\Ucal_S}^{\tilde{\beta}}=\mathfrak{p}_{S,\:\Ucal_S}^\beta,
\]
and since $\beta\in\A_{K,S_0}$, we have that $\beta\in\mathfrak{p}$, which concludes the proof.
\end{proof}

As a consequence of Proposition~\ref{pro:caracprimesAX}, we have that any direct system of prime ideals that gives an ideal in $\A_K$ contained in $M_{\Ucal}$ is uniquely determined by the ideal in $\A_{K,S_0}$ in said system, and the prime ideal $\mathfrak{p}\subseteq M_{\Ucal}$ is
\[
	\mathfrak{p}=\underset{\beta\in\mathfrak{p}}\bigcup \:\mathfrak{p}_{\Ucal}^\beta=\underset{\beta\in\mathfrak{p}\cap\A_{K,S_0}}\bigcup \!\!\!\!\! \mathfrak{p}_{\Ucal}^{\beta}.
\]

In summary, we have:

\begin{theorem}\label{thm:specAK}
Let $K$ be a number field and $S_0$ a finite subset of $X_K$ containing $X_{K,\infty}$. Let $\Omega_{S_0}=X_K\setminus S_0$. The prime ideals of $\A_K$ are completely described as follows:
\begin{enumerate}[label=(\arabic*)]
\item The maximal ideals are of one of two kinds:
\begin{enumerate}[label=\alph*.]
\item If $\upsilon_0\in X_{K,f}$,
	\[
		M_{\upsilon_0}=\{\alpha\in\A_K\:\vert\:\alpha_{\upsilon_0}=0\},
	\]
which is also a minimal prime ideal.
\item If $\Ucal$ is a non-principal ultrafilter in $\Omega_{S_0}$,
	\[
		M_{\Ucal}=\{\alpha\in\A_K\:\vert\:\{\upsilon\in\Omega_{S_0}\:\vert\:\alpha_{\upsilon}\in\m_{\upsilon}\}\in\Ucal\}
	\]
\end{enumerate}
\item The minimal prime ideals are of one of two kinds:
\begin{enumerate}[label=\alph*.]
\item If $\upsilon_0\in X_{K,f}$,
	\[
		M_{\upsilon_0}=\{\alpha\in\A_K\:\vert\:\alpha_{\upsilon}=0\},
	\]
which is also a maximal prime ideal.
\item If $\Ucal$ is a non-principal ultrafilter in $\Omega_{S_0}$,
	\[
		m_{\Ucal}=\{\alpha\in\A_K\:\vert\:\{\upsilon\in\Omega_{S_0}\:\vert\:\alpha_{\upsilon}=0\}\in\Ucal\},
	\]
which is the only minimal prime ideal contained in $M_{\Ucal}$ and is only contained in the maximal ideal $M_{\Ucal}$.
\end{enumerate}
\item If $\mathfrak{p}\in\A_K$ is any prime ideal that isn't maximal or minimal, there is exactly one non-principal ultrafilter $\Ucal$ in $\Omega_{S_0}$ such that
	\[
		m_{\Ucal}\subset\mathfrak{p}\subset M_{\Ucal},
	\]
and
	\[
		\mathfrak{p}=\underset{\beta\in\mathfrak{p}\cap\A_{K,S_0}}\bigcup \!\!\!\!\! \mathfrak{p}_{\Ucal}^{\beta},
	\]
where
	\[
		\mathfrak{p}_{\Ucal}^{\beta}=\left\{\alpha\in\A_K\:\vert\:\exists Y\in\Ucal,\:n\in\mathbb{N}\text{ s.t. }\upsilon(\alpha_{\upsilon}^n)\geq \upsilon(\beta_{\upsilon})\:\:\:\forall\upsilon\in Y\right\}
	\]
\end{enumerate}
\end{theorem}

Observe that our description of $\Spec(\A_K)$ depends on choosing an arbitrary $S_0\subseteq X_K$ at the beginning, and the ultrafilters in our final results are ultrafilters on $\Omega_{S_0}=X_K\setminus S_0$. However, by Lemma~\ref{lem:ultrafiltersfinitepoints}, since any other choice of $S_0$ differs from the first in at most a finite number of points, this will not make a difference in the description of ultrafilters on $\Omega_{S_0}$. Hence the results given do not depend on the choice of $S_0$.

\section{Closed prime ideals}
\label{sec:clopri}

We can recover the completed fields $K_{\upsilon}$ from the adele ring by taking its quotients by the principal prime ideals: it is very easy to check that, if $\upsilon_0\in X_K$, $\A_K/M_{\upsilon_0}\simeq K_{\upsilon_0}$, since
\[
	\alpha-\beta\in M_{\upsilon_0}\Longleftrightarrow (\alpha-\beta)_{\upsilon_0}=0\Longleftrightarrow \alpha_{\upsilon_0}=\beta_{\upsilon_0}.
\]

There is a reason that these ideals work better than the ones given by non-principal ultrafilters, and it is a topological reason. In \cite{Conn}, Connes proves the following result (stated as \textbf{proposition 7.2.}):

\begin{proposition}
There is a one to one correspondence between subsets $Z\subseteq X_K$ and closed ideals of $\A_K$ given by
\[
	Z \longmapsto \mathcal{I}_Z=\{\alpha\in\A_K \:\vert\: \alpha_{\upsilon}=0\:\:\:\forall \upsilon\in Z\}.
\]
\end{proposition}

Given our description of the prime ideals of $\A_K$, this implies that the only closed prime ideals in $\A_K$ are the maximal prime ideals $M_{\upsilon_0}$, for $\upsilon_0\in X_K$ (since $\mathcal{I}_Z\subseteq M_{\upsilon}$ for any $\upsilon\in Z$).

Connes stops his exploration of $\Spec(\A_K)$ there. However, our previous exhaustive description of $\Spec(\A_K)$ allows us to prove the following:

\begin{proposition}\label{densepri}
If $\Ucal$ is a non-principal ultraflter in $\Omega_{S_0}=X_K\setminus S_0$, the minimal prime ideal $m_{\Ucal}$ defined by $\Ucal$ is dense in $\A_K$. Consequently, every prime ideal of $\A_K$ not of the form $M_{\upsilon_0}$ for $\upsilon_0\in X_K$ is dense in $\A_K$, because its closure is an ideal that contains $1$.
\end{proposition}

\begin{proof}
Take the net $\{\alpha^V\}_{V\in\Ucal}$, where
\[
	\alpha^V_{\upsilon}=\begin{cases}
	0\qquad\text{ if }\upsilon\in V &\\
	1\qquad\text{ if }\upsilon\notin V\end{cases}
\]
and $\alpha^{V_1}\geq\alpha^{V_2}\Leftrightarrow V_1\subseteq V_2$. It is clear from the definition of $m_{\Ucal}$ that $\alpha^V\in m_{\Ucal}$ for every $V\in \Ucal$.

Let $N$ be an open neighbourhood of $1\in\A_K$. We can assume, without loss of generality, that it is a basic open set, that is, of the form $N=\prod N_{\upsilon}$, where each $N_{\upsilon}$ is an open neighbourhood of $1\in K_{\upsilon}$ and $N_{\upsilon}=\mathfrak{o}_{\upsilon}$ for almost every $\upsilon$. If we take $V_0=\{\upsilon\in \Omega_{S_0} \:\vert\: N_{\upsilon}=\mathfrak{o}_{\upsilon}\}$, we have:

\begin{enumerate}
\item $V_0\in\Ucal$, because $\Ucal$ contains every cofinite subset of $\Omega_{S_0}$.

\item If $\alpha^V\geq\alpha^{V_0}$ (that is, $V\subseteq V_0$) for some $V\in\Ucal$, then $\alpha^V\in N$.
\end{enumerate}

This means that the net $\{\alpha^V\}_{V\in\Ucal}$ has limit $1\in\A_K$, and therefore $m_{\Ucal}$ is dense in $\A_K$.

Now, since every prime ideal not of the form $M_{\upsilon_0}$ contains $m_{\Ucal}$ for some non-principal ultrafilter $\Ucal$ in $\Omega_{S_0}$, it must also be dense in $\A_K$.
\end{proof}

\section{Restriction of ultrafilters under surjective maps with uniformly bounded fibres}
\label{sec:ultrafiltmap}

In the next sections we will examine how, for a finite extension of number fields $L/K$, the prime ideals of the $\A_K$-algebra $\A_L$ behave relative to those of $\A_K$. Before that we need to develop some results about the behaviour of ultrafilters under extension.

Let $X$ and $Y$ be two sets and $\pi:Y\rightarrow X$ a map between them. Extend $\pi$ to a map from $\mathcal{P}(Y)$ to $\mathcal{P}(X)$ via direct images, $\pi(W)=\{\pi(y)\:\vert\: y\in W\}$. Then we can define, for an ultrafilter $\tilde{\Ucal}$ on $Y$,
\[
	  \pi\left(\tilde{\Ucal}\right)
	= \left\{\pi(W)\:\big\vert\: W\in\tilde{\Ucal}\right\}.
\]
It is not true, in general, that $\pi(\tilde{\Ucal})$ is an ultrafilter on $X$, and it may not even be a filter. This changes if $\pi$ is surjective.

\begin{proposition}\label{pro:ultramap1}
Let $X,Y$ be two sets and $\pi:Y\rightarrow X$ a surjective map. For every ultrafilter $\tilde{\Ucal}$ on $Y$, $\pi(\tilde{\Ucal})$ is an ultrafilter on $X$. Furthermore, for every ultrafilter $\Ucal$ on $X$ there is at least one ultrafilter $\tilde{\Ucal}$ on $Y$ such that $\Ucal=\pi(\tilde{\Ucal})$. 
\end{proposition}

\begin{proof}
It is easy to check that $\pi(\tilde{\Ucal})$ is an ultrafilter on $X$ if $\tilde{\Ucal}$ is an ultrafilter on $Y$. Now, let $\Ucal$ be an ultrafilter on $X$. For every $x\in X$, we fix $x^{(1)}\in\pi^{-1}(x)$. For $V\subseteq X$, define $V^{(1)}=\{x^{(1)}\:\vert\: x\in V\}\subseteq Y$ and
\[
	  \Ucal^{(1)}
	= \{W\subseteq Y\:\vert\: \exists V\in\Ucal\text{ with }V^{(1)}\subseteq W\}.
\]
Clearly $\pi(\Ucal^{(1)})=\Ucal$ as subsets of $\mathcal{P}(X)$, so we only have to prove that $\Ucal^{(1)}$ is an ultrafilter on $Y$. It is immediate to check that $\emptyset\notin\Ucal^{(1)}$ and that $\Ucal^{(1)}$ is closed under upward inclusions. Given $V_1,V_2\subseteq X$, $V_1^{(1)}\cap V_2^{(1)}=(V_1\cap V_2)^{(1)}$, so $\Ucal^{(1)}$ is also closed under finite intersection. Now take $W\subseteq Y$ and $V=\{x\in X\:\vert\: x^{(1)}\in W\}$. Then $V^{(1)}\subseteq W$ and $(X\setminus V)^{(1)}\subseteq Y\setminus W$, and since either $V$ or $X\setminus V$ are in $\Ucal$, then either $W$ or $Y\setminus W$ are in $\Ucal^{(1)}$.
\end{proof}

\begin{lemma}\label{lem:setpartition}
Let $X$ be a set and $P_1,...,P_N\in\mathcal{P}(X)\setminus\{\emptyset\}$ a finite partition of $X$, that is, $X=P_1\cup...\cup P_N$ and $P_i\cap P_j=\emptyset$ if $i\neq j$. If $\Ucal$ is an ultrafilter in $X$, $\Ucal$ contains exactly one of the sets $P_1,...,P_N$.
\end{lemma}

\begin{proof}
If there were more than one, their intersection would be empty. If there were none, their set complements would be in $\Ucal$, and their intersection would be empty.
\end{proof}

\begin{proposition}\label{pro:ultramap2}
Let $\pi:Y\rightarrow X$ be a surjective map between two sets such that there is an $n\in\mathbb{N}$ with $\#\pi^{-1}(x)=n$ for every $x\in X$. For every ultrafilter $\Ucal$ in $X$ there are exactly $n$ ultrafilters in $Y$ that restrict to $\Ucal$ via $\pi$.
\end{proposition}

\begin{proof}
For every $x\in X$, let $\pi^{-1}(x)=\{x^{(1)},...,x^{(n)}\}$. Given an ultrafilter $\Ucal$ on $X$, let $\Ucal^{(1)},...,\Ucal^{(n)}$ be the ultrafilters in $Y$ defined as in the proof of Proposition~\ref{pro:ultramap1}. These are $n$ distinct ultrafilters in $Y$ (because $V^{(i)}\cap V^{(j)}=\emptyset$ if $i\neq j$) and $\pi(\Ucal^{(i)})=\Ucal$ for $i=1,...,n$. Let us prove that there are no others.

Let $\tilde{\Ucal}$ be an ultrafilter in $Y$ such that $\pi(\tilde{\Ucal})=\Ucal$. Let $\tilde{V}\in\tilde{\Ucal}$. By lemma \ref{lem:setpartition}, there is exactly one $i=1,...,n$ such that $X^{(i)}\in\tilde{\Ucal}$. Since
\[
	\pi(\tilde{V})^{(i)}\subseteq\tilde{V}\cap X^{(i)}\subseteq \tilde{V}
\]
and $\pi(\tilde{V})^{(i)}\in\tilde{\Ucal}^{(i)}$, we have $\tilde{V}\in\Ucal^{(i)}$, so $\tilde{\Ucal}\subseteq\Ucal^{(i)}$ and, by maximality of ultrafilters, $\tilde{\Ucal}=\Ucal^{(i)}$.
\end{proof}

\begin{proposition}\label{pro:ultramap3}
Let $\pi:Y\rightarrow X$ be a surjective map between two sets such that its fibres are uniformly bounded by $n\in\mathbb{N}$ (that is, for every $x\in X$ we have $\#\pi^{-1}(x)\leq n$). For every ultrafilter $\Ucal$ on $X$ there are at most $n$ ultrafilters on $Y$ that restrict to $\Ucal$ via $\pi$.
\end{proposition}

\begin{proof}
For $x\in X$, set $m_x=\#\pi^{-1}(x)$ and fix an ordering of the fibre of $x$, $\pi^{-1}(x)=\{x^{(1)},...,x^{(m_x)}\}$. For every $x\in X$ we complete $\pi^{-1}(x)$ to an $n$-element set
\[
	W_x=\{x^{(1)},...,x^{(m_x)},x^{(m_x+1)},...,x^{(n)}\},
\]
by adding arbitrary elements $x^{(m_x+1)},...,x^{(n)}$ (for example, copies of $x^{(1)}$). Now consider the disjoint union
\[
	\tilde{Y}=\underset{ x\in X}\bigsqcup W_x
\]
and define $p:\tilde{Y}\rightarrow Y$ as
\begin{enumerate}
\item $p(x^{(i)})=x^{(i)}$ if $i\leq m_x$.

\item $p(x^{(i)})=x^{(1)}$ if $i>m_x$.
\end{enumerate}

Let $\tilde{\pi}=\pi\circ p:\tilde{Y}\rightarrow X$ so that $\tilde{\pi}^{-1}(x)=W_x$. We can apply Proposition~\ref{pro:ultramap2} to $\tilde{\pi}$ and Proposition~\ref{pro:ultramap1} to $p$, so that, given an ultrafilter $\Ucal$ in $X$, the ultrafilters in $Y$ that restrict to $\Ucal$ via $\pi$ are the restriction to $Y$ via $p$ of the ultrafilters in $\tilde{Y}$ that restrict to $\Ucal$ via $\tilde{\pi}$, of which there are $n$.
\end{proof}

In the case of a finite extension of number fields $L/K$, this last result will help us get an upper bound for the number of prime ideals in $\A_L$ above a given prime ideal in $\A_K$. We need one more result, that is a consequence of lemma \ref{lem:setpartition}:

\begin{lemma}\label{lem:onepointperfibre}
Let $\pi:Y\rightarrow X$ as in Proposition~\ref{pro:ultramap3} and $\tilde{\Ucal}$ an ultrafilter on $Y$. For every $\tilde{W}\in\tilde{\Ucal}$ there is some $W\in\tilde{\Ucal}$ with $W\subseteq\tilde{W}$ and such that $W$ contains at most one point from every fibre of $\pi$.
\end{lemma}

\begin{proof}
Since $\tilde{\Ucal}$ is closed under finite intersections, the assertion is equivalent to the existence of a set $W\in\tilde{\Ucal}$ such that $W$ contains at most one point from every fibre of $\pi$, and we can take the set $X^{(i)}$, with $i=1,...,n$ the only index such that $X^{(i)}\in\tilde{\Ucal}$, by lemma \ref{lem:setpartition}.
\end{proof}

\section{Extension of prime ideals in finite extensions}
\label{sec:finiteext}

If $L/K$ is a finite extension of number fields of degree $n$, it is well known (see \cite{CaFr}) that $\A_L$ is a free $\A_K$-algebra of rank $n$, and in fact $\A_L\simeq \A_K\otimes_K L$. The $\A_{K}$-algebra monomorphism $\A_K\hookrightarrow \A_L$ induces a morphism
\[
	\begin{array}{ccc}
	\Spec(\A_L) & \longrightarrow & \Spec(\A_K) \\
	\mathfrak{B} & \longmapsto & \mathfrak{B}\cap\A_K.
	\end{array}
\]

We will now study the fibres of this morphism using the description of $\Spec(\A_K)$ given in section \ref{sec:primesadeles}. For simplicity we will assume that our base field is $\Q$ and consider a finite extension $K/\Q$ of degree $n\in\mathbb{N}$. Let $X_{\Q}=\{\infty\}\cup\{p\in\P\}$ be the set of places in $\Q$, with starting set $S_0=\{\infty\}$ and $\Omega = X_{\Q}\setminus S_0$. Let $X_K$ be the set of places in $K$, and $\tilde{S}_0$ the set of Archimedean places, with $\tilde{\Omega}=X_K\setminus\tilde{S}_0$. We have a restriction map $r=r_{K/\Q}:X_K\rightarrow X_{\Q}$ for which $\tilde{S}_0=r^{-1}(S_0)$. The basic theory of valuations (see \cite{CaFr}, \cite{Ram}) tells us that $r$ is a map with finite fibres, with $\#\tilde{S}_0=s_1+s_2$ (where $s_1$ is the number of real embeddings of $K$ and $s$ is the number of pairs of conjugate complex embeddings) and $\#r^{-1}(p)=g_p$ (where $g_p$ is the number of prime ideals in the decomposition of the ideal $(p)\subseteq\mathfrak{o}_K$). The map $r:\tilde{\Omega}\rightarrow\Omega$ gives us, according to Proposition~\ref{pro:ultramap3}, a surjective map between the ultrafilters on $\tilde{\Omega}$ and the ultrafilters on $\Omega$ with finite fibres uniformly bounded by $n$.

\begin{proposition}\label{pro:maxextension}
For the maximal and minimal ideals of $\A_K$, we have:
\begin{enumerate}
\item Given $\upsilon\in X_K$, $M_{\upsilon}\cap\A_{\Q}=M_{r(\upsilon)}$.

\item If $\tilde{\Ucal}$ is a non-principal ultrafilter in $\tilde{\Omega}$, $M_{\tilde{\Ucal}}\cap\A_{\Q}=M_{r(\tilde{\Ucal})}$, and $m_{\tilde{\Ucal}}\cap\A_{\Q}=m_{r(\tilde{\Ucal})}$.
\end{enumerate}

Therefore, every maximal (resp. minimal) prime ideal of $\A_{\Q}$ has some and at most $n$ maximal (resp. minimal) prime ideals in $\A_K$ lying above it.
\end{proposition}

\begin{proof}
The first part is easy to prove. For the second, using Lemma~\ref{lem:onepointperfibre} we can caracterize the elements in $M_{\tilde{\Ucal}}$ and $m_{\tilde{\Ucal}}$ using sets of $\tilde{\Ucal}$ that have at most one point in each fibre of $r$. This gives us $M_{r(\tilde{\Ucal})}\subseteq M_{\tilde{U}}\cap\A_{\Q}$, and the same for the corresponding minimal prime ideal. The opposite inclusion is trivial, and we have completed our proof.
\end{proof}

\begin{proposition}\label{pro:primeextension}
If $\mathfrak{p}$ is a prime ideal of $\A_{\Q}$ inside the maximal ideal $M_{\Ucal}$ for a non-principal ultrafilter $\Ucal$ of $\Omega$, and $\tilde{\Ucal}$ is an ultrafilter on $\tilde{\Omega}$ such that $r(\tilde{\Ucal})=\Ucal$, there is exactly one prime ideal $\tilde{\mathfrak{p}}$ in $M_{\tilde{\Ucal}}$ with $\tilde{\mathfrak{p}}\cap \A_{\Q}=\mathfrak{p}$, and $\tilde{\mathfrak{p}}$ is given by
\[
	\tilde{\mathfrak{p}}=\underset{\beta\in\mathfrak{p}}\bigcup \mathfrak{p}_{\tilde{\Ucal}}^{\beta}.
\]
\end{proposition}

\begin{proof}
It is clear that the prime ideal $\tilde{\mathfrak{p}}$ defined as above restricts to $\mathfrak{p}$. We only need to prove uniqueness, and for that we will prove that, given a prime ideal $\mathfrak{B}$ in $\A_K$ inside the maximal ideal $M_{\tilde{\Ucal}}$ for a non-principal ultrafilter $\tilde{\Ucal}$ in $\tilde{\Omega}$ and $\mathfrak{b}=\mathfrak{B}\cap\A_{\Q}$, then
\[
	\mathfrak{B}=\underset{\beta\in\mathfrak{b}}\bigcup \mathfrak{p}_{\tilde{\Ucal}}^{\beta}.
\]
The inclusion $\supseteq$ is easy, since $\mathfrak{b}\subseteq\mathfrak{B}$. Take $\Ucal=r(\tilde{\Ucal})$ and  $\tilde{\beta}\in\mathfrak{B}$. The set $\tilde{W}=\{\upsilon\in\tilde{\Omega}\:\vert\: \tilde{\beta}_{\upsilon}\in\mathfrak{m}_{\upsilon}\}$ is in $\tilde{U}$, and by Lemma~\ref{lem:onepointperfibre} there is a set $W$ in $\tilde{U}$ inside $\tilde{W}$ having at most one point in each fibre of $r$. Take $V=r(W)$ and $\beta\in\A_{\Q}$ defined by
\begin{enumerate}
\item $\beta_u=\tilde{\beta}_{\upsilon}$ if $u\in V$ and $\upsilon$ is the point in $W$ above $u$.

\item $\beta_u=0$ if $u\notin V$.
\end{enumerate}
It is straightforward to check that $\beta\in\mathfrak{b}$ and $\mathfrak{p}_{\tilde{\Ucal}}^{\tilde{\beta}}=\mathfrak{p}_{\tilde{\Ucal}}^{\beta}$ (because $\beta$ and $\tilde{\beta}$ have the same components in a set of $\tilde{\Ucal}$), so we have the equality we wanted.
\end{proof}

\begin{corollary}\label{cor:fibreprime}
If $\Q\hookrightarrow K$ is a finite extension of degree $n\in\mathbb{N}$, the morphism $\Spec(\A_K)\rightarrow\Spec(\A_{\Q})$ is surjective and has finite fibres, uniformly bounded by $n$.
\end{corollary}

As we stated at the beginning of the section, the base field need not be $\Q$, and the same results, with the same proofs, are valid for any finite extension of number fields $L/K$.

\section{Generalizations}
\label{sec:general}

Throughout this paper we have worked in the adele ring of an algebraic number ring. However, the techniques we have used and the results obtained can be used in more general settings.

Take $K$ to be the function field of an algebraic curve over a finite field, that is, a finite extension of an extension of trascendence degree $1$ of a finite field $\mathbb{F}_q$. These fields, together with algebraic number fields, are called \textit{global fields}, and are studied together because they share many similar properties: primarily, they are the only two kinds of fields with finite residue fields at the non-Archimedean places (see \cite{ArtinW}). The adele ring $\A_K$ of these function fields $K$ can be defined in the exact same way as that of a number field, and its study is also mostly the same (see \cite{Cass}, \cite{Ram} or \cite{TaPhD}). This means that $\Spec(\A_K)$ can be computed just as in the case of a number field, and we obtain the same results.

We can also consider the adele ring of function fields of curves that are not necessarily over finite fields (see \cite{Porr}), for example). If $k$ is a perfect field and $X$ a smooth, complete and connected curve over $k$, we can define the adele ring $\A_X$ of $X$ as the adele ring of its function field $\Sigma_X$. If $k$ is not a finite field, this adele ring $\A_X$ differs slightly but significantly from the global field case: the most important difference is that, topologically, $\A_X$ is no longer locally compact. However, the similarities are enough for our treatment of the prime ideals to apply here as well, so we can effectively describe $\Spec(\A_X)$, and the results about extensions and topology are still true, because they do not require local compactness.

\end{document}